\documentclass[12pt]{amsart}
\usepackage[pagewise]{lineno}
\usepackage{amscd,amssymb}
\usepackage[arrow,matrix]{xy}
\usepackage[plainpages,backref,urlcolor=blue]{hyperref}
\usepackage{mathdots}
\usepackage{mathrsfs}
\usepackage{amsthm}
\usepackage{bm}
\usepackage{amscd}
\usepackage{mathtools}

\theoremstyle{plain}
\newtheorem{thm}[subsection]{Theorem}
\newtheorem{lem}[subsection]{Lemma}
\newtheorem{prop}[subsection]{Proposition}
\newtheorem{cor}[subsection]{Corollary}
\newtheorem{claim}[subsection]{Claim}

\theoremstyle{definition}

\newtheorem{definition}[subsection]{Definition}
\newtheorem{question}[subsection]{Question}
\newtheorem{ex}[subsection]{Example}

\newcommand{\bb}{\mathbb}
\newcommand{\ov}{\overline}
\newcommand{\p}{\partial}

\newcommand{\ord}{\mathrm{ord}}
\newcommand{\mult}{\mathrm{mult}}
\newcommand{\Supp}{\mathrm{Supp}}

\begin{document}
	\date{}
		
	\title[Monotonic invariants under blowups]{Monotonic invariants under blowups}
		
	\author[ZHENJIAN WANG]{ ZHENJIAN WANG  }
	\address{YMSC, Tsinghua University, 100084 Beijing, China}
	\email{wzhj@mail.tsinghua.edu.cn}
		
		
		
	\begin{abstract}
	We prove that the numerical invariant $3\mu-4\tau$ of a reduced irreducible plane curve singularity germ is non-negative, non-decreasing under blowups and strictly increasing unless the curve is non-singular. This provides a new perspective to understand the question posed by A. Dimca and G.-M.~Greuel. Moreover, our work can be put in the general framework of discovering monotonic invariants under blowups.
	\end{abstract}
	\maketitle
	
\section{Introduction}
Blowups are very important tools in algebraic geometry. We can obtain a resolution of a singular complex algebraic variety using successive blowups, and this explains to some extent the remarkable role of blowups in singularity theory; see the celebrated paper by H.~Hironaka \cite{Hir} or the exposition by J. Koll\'ar \cite{Kol}.
The most extensively studied case is for plane curves, see for instance, \cite[Chapter 1]{Kol}, or \cite[Chapter 8]{Mum}, or \cite[Chapter 3]{Wal}, or \cite[Chapter 4]{Alv}.
Among the key reasons for our success in resolving singularities using blowups is the fact that some numerical invariants decrease or increase when a blowup is performed.
It is this monotonic nature of these numerical invariants that ensures that our blowup sequence will terminate after a finite number of steps.

For a reduced irreducible plane curve $C$, the numerical invariants we often use are the Milnor number $\mu(C)$ and the Tjurina number $\tau(C)$; for their definitions, see Section \ref{sec1} below. They both decrease if we perform one blowup, that is, if $\widetilde{C}$ is the strict transform of $C$ under a single embedded blowup, then $\mu(\widetilde{C})\leq\mu(C)$ and $\tau(\widetilde{C})\leq\tau(C)$, and the equality holds only if $C$ is smooth. In this article, we will show that their combination $3\mu-4\tau$ however changes in the opposite way under a blowup. We will prove the following result.

\begin{thm}\label{thm1}
Let $(C,0)\subset(\bb{C}^2,0)$ be a reduced irreducible plane curve germ, and $\widetilde{C}$  the strict transform of $C$ under a single embedded blowup at $0\in\bb{C}^2$. Then the following hold:
$$
3\mu(\widetilde{C})-4\tau(\widetilde{C})\geq 3\mu(C)-4\tau(C).
$$
Moreover, the equality holds if and only if $C$ is smooth.
\end{thm}

A direct consequence is a positive answer to the irreducible case of the following question posed by A.~Dimca and G.-M.~Greuel in \cite{DGr}.

\begin{question}\label{ques1}
{\it Is it true that $\frac{\mu}{\tau}<\frac{4}{3}$ for any isolated plane curve singularity ? }
\end{question}

In fact, one of the motivations of our work is to answer this question by showing that $3\mu-4\tau$ has some monotonicity along a blowup sequence; see the end of Section \ref{sec4} in the sequel.
We are informed that two groups of mathematicians have already solved the irreducible case independently; see \cite{MPGM} and \cite{GHe}. Our proof of Theorem \ref{thm1} relies on Theorem 13 in \cite{GHe}, and consequently, we will not give an independent solution.
But as one can easily see, Theorem \ref{thm1} is stronger than an affirmative answer to Question \ref{ques1}. In addition, we will avoid applying the formula for the
generic dimension of the moduli space obtained in \cite{Gen}, which is a crucial ingredient in both \cite{MPGM} and \cite{GHe}. Furthermore, although a positive answer is already known for reducible semiquasihomogeneous singularities (see \cite{AlB}), the general case is still open; we hope our method will help to tackle it.

Another motivation of our work is an attempt to borrow useful ideas from differential geometry to study singularities. In the analysis of a geometric flow, some increasing or decreasing quantities along the flow are very useful and even indispensable. A typical example is the Ricci flow. In the proof of the famous Poincar\'e Conjecture, one of the most outstanding achievements of G.~Perelman is to discover some monotonic quantities under the Ricci Flow, with the help of which non-collapsing results can be obtained. For more details about Perelman's contributions as well as a complete proof of the Poincar\'e Conjecture, we refer the interested reader to \cite{MT}. Similarly, as mentioned at the beginning of this note, we also need to discover monotonic invariants in order to have a control of the blowup sequences when we construct resolutions of singularities. If we see a blowup sequence denoted by
$$
\cdots\to X_N\to X_{N-1}\to\cdots\to X_2\to X_1\to X_0=X
$$
as a ``discrete flow" of singularities, then our search of monotonic numerical invariants is natural and indispensable as in the framework of geometric flow in differential geometry. We hope this new perspective will help in our future analysis of blowup sequences.

In the last part of this paper, we will consider singularity comparison for which monotonic invariants under blowups are very useful.
Given an interesting class of singularities, we may compare two of them and say one is smoother or more singular than another. We say a singularity $X'$ is smoother than another $X$ if $X'$ can be obtained by a sequence of successive blowups of $X$. In practice, it is difficult to check whether a singularity $X'$ is smoother than $X$; but if we could find a numerical invariant, denoted by $I$, that increases under a blowup, then definitely, we have that $X'$ is not smoother than $X$ if $I(X')<I(X)$. We will explain more details and give some examples to illustrate the basic ideas in the last section. Our attempt is just a beginning, and hopefully more results will be obtained in the future.

\bigskip
We would like to thank Professor A. Dimca for informing us of the latest development on Question \ref{ques1}, and Patricio Almir\'on for his remarks on the previous version of our manuscript. We thank Yau Mathematical Sciences Center for their financial support and wonderful working conditions. We also thank an anonymous referee for useful comments and suggestions.

\section{Numerical invariants for plane curves}\label{sec1}
All plane curve germs in the sequel will be reduced and irreducible unless stated otherwise; they are also called plane branches. We will not distinguish between a complex space and its representatives and we are free to choose the representative as small as we want when necessary.

In this section we will introduce the basic notions and notation, especially, the numerical invariants of a plane curve. We refer to \cite{Wal}, \cite{GLS} and \cite{Zar} for more details.

\subsection{Invariants}
A  germ of plane curve $(C,0)\subset(\bb{C}^2,0)$ is defined by an irreducible reduced holomorphic function germ $f\in\bb{C}\{x,y\}$ such that $f(0,0)=0$.
The Puiseux parametrization for $(C,0)$ is given by
$$
\left\{\begin{aligned}
x&=t^m,\\
y&=s(t)=t^{\beta_1}+\sum_{j>\beta_1}a_jt^j,\qquad m<\beta_1,m\nmid\beta_1,
\end{aligned}\right.
$$
where $s(t)\in\bb{C}\{t\}$. The Puiseux characteristic $(m;\beta_1,\cdots,\beta_g)$ is defined as follows: denote by $\beta_0=m$ and $e_0=m$. For $i\geq 1$, define $$
\beta_i=\min\{j\,:\,a_j\neq 0\ \text{and }e_{i-1}\nmid j\}
$$
 and $e_i=\gcd(e_{i-1},\beta_i)$.

We call $m$ the \emph{multiplicity} of $C$ as well as $f$. The Puiseux series given as above, we can assume that $f$ is a Weierstrass polynomial by the Weierstrass preparation theorem (see \cite[Theorem 1.6, p.\,11]{GLS}), that is, $f$ is of the following form
\begin{equation}\label{eqC}
f(x,y)=y^m+a_1(x)y^{m-1}+\cdots+a_m(x),
\end{equation}
where $a_i(x)\in\bb{C}\{x\}$ and $\ord_x(a_i(x))\geq i$ for $i=1,\cdots, m$. Here we have used the fact that the curve germ $C$ is irreducible and that $m$ is the multiplicity.

The Milnor number of $C$ is defined by
$$
\mu(C)=\dim_\bb{C}\frac{\bb{C}\{x,y\}}{(\p f/\p x,\p f/\p y)},
$$
and the Tjurina number of $C$ defined by
$$
\tau(C)=\dim_\bb{C}\frac{\bb{C}\{x,y\}}{(f, \p f/\p x,\p f/\p y)}.
$$

Let $\nu: (\ov{C},0)\to (C,0)$ be the normalization. As shown in \cite[Theorem 1.2 (2) and Remark 3.1]{DGr}, we have
\begin{equation}\label{dif}
\tau(C)=\frac{\mu(C)}{2}+\dim_{\bb{C}}\frac{\Omega_{\ov{C}}}{\nu^*\Omega_{C}},
\end{equation}
where $\Omega_{C}$ is the module of K\"ahler differentials on $C$; similarly for $\Omega_{\ov{C}}$.

\subsection{Changes under a blowup}
Let $\pi: \widetilde{C}\to C$ be the strict transform of $C$ under an embedded blowup of $\bb{C}^2$ at $0$. With the help of local coordinates, the blowup can be represented as
$$
\left\{\begin{aligned}
x&=x_1,\\
y&=x_1y_1,
\end{aligned}\right.
$$
so the local defining function of $\widetilde{C}$ in the local coordinates $(x_1,y_1)$ is given by
\begin{equation}\label{eqtildeC}
\widetilde{f}(x_1,y_1)=x_1^{-m}f(x_1,x_1y_1)=y_1^m+\frac{a_1(x_1)}{x_1}y_1^{m-1}+\cdots+\frac{a_m(x_1)}{x_1^m}.
\end{equation}

\begin{definition}
We call $\pi: \widetilde{C}\to C$ or simply $\widetilde{C}$ a \emph{blowup} of $C$.
\end{definition}

It is well-known that $\mu(C)-\mu(\widetilde{C})=m(m-1)$; see \cite[Theorem 6.59, p.\,149]{Wal}. Moreover, let $\widetilde{\nu}:(\ov{C},0)\to(\widetilde{C},0)$ be the normalization. By \eqref{dif}, applied to $\widetilde{C}$ instead of $C$, we have
$$
\tau(\widetilde{C})=\frac{\mu(\widetilde{C})}{2}+\dim\frac{\Omega_{\ov{C}}}{\widetilde{\nu}^*\Omega_{\widetilde{C}}}.
$$
Thus,
\begin{eqnarray}\label{taudif}
\tau(C)-\tau(\widetilde{C})&=&\frac{\mu(C)-\mu(\widetilde{C})}{2}+\dim\frac{\widetilde{\nu}^*\Omega_{\widetilde{C}}}{\nu^*\Omega_{C}}\nonumber\\
                       &=&\frac{m(m-1)}{2}+\dim\frac{\widetilde{\nu}^*\Omega_{\widetilde{C}}}{\nu^*\Omega_{C}}.
\end{eqnarray}

\section{Equisingular deformations of a curve}\label{sec2}
In this section, we consider deformations of a curve germ. For the basic theory, we refer the reader to \cite{Tei} and \cite[Chapter 2]{GLS}.

\subsection{Definition}
Given a  plane curve germ $(C,0)\subset(\bb{C}^2,0)$, we will focus on equisingular deformations of $(C,0)$; see \cite[Chapter 2, Definition 2.6, p.\,271]{GLS}.  A deformation $(C,0)\xhookrightarrow{i}(\mathscr{C},x_0)\xrightarrow{\phi}(T,t_0)$ is called an \emph{equisingular deformation} if there exists a section $\sigma: (T,t_0)\to(\mathscr{C},0)$ such that the following holds. There exist small representatives of $(i,\phi,\sigma)$ and a commutative diagram of morphisms of complex spaces
\begin{equation}\label{equidef}
\xymatrix{
\mathscr{C}^{(N)}\ar[r]\ar@{^{(}->}[dd]& \mathscr{C}^{(N-1)}\ar[r]\ar@{^{(}->}[dd]& \cdots\ar[r]          & \mathscr{C}^{(0)}\ar[rd]^-\phi\ar@{^{(}->}[dd] &            \\
                                       &                                        &                       &                               &     T      \\
\mathscr{M}^{(N)}\ar[r]^-{\pi_N}        &\mathscr{M}^{(N-1)}\ar[r]^-{\pi_{N-1}}   & \cdots \ar[r]^-{\pi_1} & \mathscr{M}^{(0)}\ar[ru]      &             \\
M^{(N)}\ar[r]\ar@{^{(}->}[u]           &M^{(N-1)}\ar[r]\ar@{^{(}->}[u]          & \cdots \ar[r]         & M^{(0)}\ar[r]\ar@{^{(}->}[u]  &  \{t_0\}\ar@{^{(}->}[uu]
}
\end{equation}
together with pairwise disjoint sections
$$
\sigma_1^{(\ell)},\cdots,\sigma_{k_\ell}^{(\ell)}\,:\,T\to\mathscr{C}^{(\ell)}\subset\mathscr{M}^{(\ell)},\qquad\ell=0,\cdots,N
$$
of the composition $\mathscr{M}^{(\ell)}\xrightarrow{\pi_\ell}\mathscr{M}^{(\ell-1)}\xrightarrow{\pi_{\ell-1}}\cdots\xrightarrow{\pi_1}\mathscr{M}^{(0)}\to T$ with the following properties:

(i) The lowest row of \eqref{equidef} induces a minimal embedded resolution of the
plane curve germ $(C,0)\subset(M^{(0)},0)=(\bb{C}^2,0)$.

(ii) For $\ell=0$, we have $(\mathscr{M}^{(0)},x_0)=(\bb{C}^2\times T,(0,t_0))$, $(\mathscr{C}^{(0)}, x_0)=(\mathscr{C},x_0)$,
$k_0=1$. Moreover, $\sigma_1^{(0)}:T\to\mathscr{M}^{(0)}$ is the section (induced by) $\sigma$, and
$(\mathscr{C}^{(0)},x_0)\hookrightarrow(\mathscr{M}^{(0)},x_0)\to(T,t_0)$ defines an equimultiple (embedded) deformation
of $(C,0)$ along $\sigma_1^{(0)}$.

(iii) For $\ell=1$, we have that $\pi_1:\mathscr{M}^{(1)}\to\mathscr{M}^{(0)}$ is the blowup of $\mathscr{M}^{(0)}$ along
the section $\sigma_1^{(0)}$, $\mathscr{C}^{(1)}$ is the strict transform of $\mathscr{C}^{(0)}\subset\mathscr{M}^{(0)}$, and $\mathscr{E}^{(1)}$ is
the exceptional divisor of $\pi_1$.

(iv) For $\ell\geq 1$, we require inductively that

 (iv-1) $\sigma_1^{(\ell)}(t_0),\cdots,\sigma_{k_\ell}^{(\ell)}(t_0)$ are precisely the non-nodal singular points of the
reduced total transform of $(C,0)\subset(M^{(0)},0)=(\bb{C}^2,0)$.

(iv-2) $\mathscr{C}^{(\ell)}\cup\mathscr{E}^{(\ell)}\hookrightarrow\mathscr{M}^{(\ell)}\to T$ induces (embedded) equimultiple deformations
along $\sigma_1^{(\ell)},\cdots,\sigma_{k_\ell}^{(\ell)}$, of the respective germs of the reduced total
transform $C^{(\ell)}\cup E^{(\ell)}$ of $(C,0)$ in $M^{(\ell)}$.

(iv-3) The sections are compatible, that is, for each $j = 1,\cdots, k_\ell$, there is some
$1\leq i\leq k_{\ell-1}$ such that $\pi_\ell\circ\sigma_j^{(\ell)}=\sigma_i^{(\ell-1)}$.

(iv-4) $\pi_{\ell+1}:\mathscr{M}^{(\ell+1)}\to\mathscr{M}^{(\ell)}$ is the blowing up of $\mathscr{M}^{(\ell)}$ along $\sigma_1^{(\ell)},\cdots,\sigma_{k_\ell}^{(\ell)}$,
$\mathscr{C}^{(\ell+1)}$ is the strict transform of $\mathscr{C}^{(\ell)}\subset\mathscr{M}^{(\ell)}$, and $\mathscr{E}^{(\ell+1)}$ is the exceptional
divisor of the composition $\pi_1\circ\pi_2\circ\cdots\circ\pi_{\ell+1}$.

\subsection{Properties of equisingular deformations}
An equisingular deformation of a curve germ is a deformation of one of the following equivalent forms:

(i) a $\mu$-constant deformation, see \cite[Corollary 2.68, p.\,371]{GLS}.

(ii) a topologically trivial deformation, see \cite{Tim},\cite{Kin} and \cite{TrR}.

(iii) a constant semigroup deformation, see \cite{Tei}.

Moreover, the equisingular deformation of a curve is unobstructed and has a nonsingular miniversal base space, see \cite[3.1.1, Corollary 1, p.\,132]{Tei}. For the definitions of versal and miniversal deformations, we refer the reader to \cite[Definition 1.8, p.234]{GLS}.

In the sequel, we will assume that the deformation $(C,0)\xhookrightarrow{i}(\mathscr{C},x_0)\xrightarrow{\phi}(T,t_0)$ is the miniversal deformation; as in \cite{Tei}, we use the notation $\tau_-(C)=\dim T$.
By \cite[3.1.2, Corollary 2, p.\,132]{Tei}, there exists a section $\sigma: (T,t_0)\to(\mathscr{C},x_0)$ such that it picks out the unique singular point of each fiber of $\phi$.

From $(i,\phi,\sigma)$, we can obtain the commutative diagram \eqref{equidef}. We will use the notation
$$
\ov{\mathscr{C}}=\mathscr{C}^{(N)},\ \widetilde{\mathscr{C}}=\mathscr{C}^{(1)}.
$$
We have natural morphisms
$$
\nu: \ov{\mathscr{C}}\to\mathscr{C},\ \widetilde{\nu}: \ov{\mathscr{C}}\to\widetilde{\mathscr{C}},\ \pi: \widetilde{\mathscr{C}}\to\mathscr{C},
$$
$$
\widetilde{\phi}: \widetilde{\mathscr{C}}\to T\text{ and } \ov{\phi}: \ov{\mathscr{C}}\to T.
$$

\begin{definition}
We call $\pi: \widetilde{\mathscr{C}}\to\mathscr{C}$ or simply $\widetilde{\mathscr{C}}$ an \emph{equisingular blowup} of $\mathscr{C}$.
\end{definition}

Recall that $\widetilde{C}$ is a blowup of $C$ and we assumed that $\mathscr{C}\xrightarrow{\phi} (T,t_0)$ is a miniversal deformation of $C$. By \cite[Remark 2.61 (5), p.\,273 and Proposition 2.13, p.\,282]{GLS}, we have that $\widetilde{\mathscr{C}}\xrightarrow{\widetilde{\phi}} (T,t_0)$ is an equisingular deformation of $\widetilde{C}$.

\begin{prop}\label{versal}
The deformation $\widetilde{\mathscr{C}}\xrightarrow{\widetilde{\phi}} (T,t_0)$ is a versal equisingular deformation of $\widetilde{C}$.
\end{prop}
\begin{proof}
If $\widetilde{C}$ is smooth, the conclusion trivially holds. So we will focus on the case where $\widetilde{C}$ is not smooth.

Suppose the equation of $C$ is given in \eqref{eqC}, then the local defining function of $\tilde{C}$ is given as in \eqref{eqtildeC} by
$$
\widetilde{f}(x_1,y_1)=y_1^m+b_1(x_1)y_1^{m-1}+\cdots+b_m(x_1).
$$
Hence $1,y_1,\cdots,y_1^{m-1}$ gives a system of generators over $\bb{C}\{x_1\}$ of the Tjurina algebra $\bb{C}\{x_1,y_1\}/(\widetilde{f},\p\widetilde{f}/\p x_1,\p\widetilde{f}/\p y_1)$ of $\tilde{C}$. By definition
$$
\dim\frac{\bb{C}\{x_1,y_1\}}{(\widetilde{f},\p\widetilde{f}/\p x_1,\p\widetilde{f}/\p y_1)}=\tau(\widetilde{C}),
$$
so there exist $\tau(\widetilde{C})$ elements of $\bb{C}\{x_1\}[y_1]$ of the following form
$$
p_i(x_1,y_1)=p_{i,0}(x_1)y_1^{n_i}+p_{i,1}(x_1)y_1^{n_i-1}+\cdots+p_{i,n_i}(x_1),\ 0\leq n_i\leq m-1, i=1,\cdots,\tau(\widetilde{C}),
$$
such that they represent a basis of the vector space $\bb{C}\{x_1,y_1\}/(\widetilde{f},\p\widetilde{f}/\p x_1,\p\widetilde{f}/\p y_1)$. By \cite[Theorem 1.16, p.\,238]{GLS}, the miniversal deformation of $\widetilde{C}$ is given by
$$
(\widetilde{C},0)\hookrightarrow(\mathscr{X},0)\subset(\bb{C}^2,0)\times(\bb{C}^{\tau(\widetilde{C})},0)\rightarrow(\bb{C}^{\tau(\widetilde{C})},0),
$$
where
$$
\mathscr{X}=\{((x_1,y_1),{\bf s})\in(\bb{C}^2,0)\times(\bb{C}^{\tau(\widetilde{C})},0)\,:\, \widetilde{F}(x_1,y_1,{\bf s})=0\,\},
$$
and $\widetilde{F}(x_1,y_1,{\bf s})=\widetilde{f}_{{\bf s}}(x_1,y_1)$,
$$
\widetilde{f_{\bf s}}(x_1,y_1)=\tilde{f}(x_1,y_1)+\sum_{i=1}^{\tau(\widetilde{C})}s_ip_i(x_1,y_1),\ \text{for }{\bf s}=(s_1,\cdots,s_{\tau(\widetilde{C})})\in(\bb{C}^{\tau(\widetilde{C})},0).
$$

The rest of the proof will be divided into three steps.

{\bf Step 1: Miniversal deformation of $\widetilde{C}$. }
Let $(\widetilde{C},0)\hookrightarrow(\mathscr{Y},y_0)\rightarrow(\widetilde{T},\widetilde{t}_0)\simeq(\bb{C}^{\tau_-(\widetilde{C})},0)$ be the miniversal equisingular deformation of $\widetilde{C}$. Since $\mathscr{Y}\to\widetilde{T}$ is a priori a deformation of $\widetilde{C}$, by the versality of $\mathscr{X}\to(\bb{C}^{\tau(\widetilde{C})},0)$, there exists a morphism
\begin{equation}\label{eqpi}
{\bf\varphi}=(\varphi_1,\cdots,\varphi_{\tau(\widetilde{C})}): (\widetilde{T},\widetilde{t}_0)\to(\bb{C}^{\tau(\widetilde{C})},0)
\end{equation}
under which $\mathscr{Y}\to\widetilde{T}$ is the pullback of $\mathscr{X}\to(\bb{C}^{\tau(\widetilde{C})},0)$.
Therefore, the miniversal equisingular deformation $\mathscr{Y}\to\widetilde{T}$ can be represented as
$$
\mathscr{Y}=\{((x_1,y_1),\widetilde{\bf t})\in(\bb{C}^2,0)\times(\widetilde{T},\widetilde{t}_0)\,:\, \widetilde{F}(x_1,y_1,\widetilde{\bf t})=0\,\}
$$
where $\widetilde{F}(x_1,y_1,\widetilde{\bf t})=\widetilde{f_{\widetilde{\bf t}}}(x_1,y_1)$,
\begin{equation}\label{eqtildeft}
\widetilde{f_{\widetilde{\bf t}}}(x_1,y_1)=\tilde{f}(x_1,y_1)+\sum_{i=1}^{\tau(\widetilde{C})}\varphi_i(\widetilde{\bf t})p_i(x_1,y_1),\ \text{for }\widetilde{\bf t}\in(\widetilde{T},\widetilde{t}_0).
\end{equation}
Denote by $\widetilde{C}_{\widetilde{\bf t}}\subset\bb{C}^2$ the plane curve germ defined by the equation $\widetilde{f_{\widetilde{\bf t}}}(x_1,y_1)=0$, for $\widetilde{{\bf t}}\in\widetilde{T}$.

{\bf Step 2: Blowing down deformations of $\widetilde{C}$. }
By \cite[3.12, Corollary 2, p.\,132]{Tei}, we have a section of the deformation $(\mathscr{Y},y_0)\rightarrow(\widetilde{T},\widetilde{t}_0)$ which picks out the unique singular point of each fiber. So there exists a morphism
$$
\widetilde{\sigma}=(\widetilde{\sigma}_1,\widetilde{\sigma}_2): (\widetilde{T},\widetilde{t}_0)\to (\bb{C}^2,0)
$$
such that $\widetilde{\sigma}(\widetilde{\bf t})$ is the unique singular point of $\widetilde{C}_{\widetilde{\bf t}}$. Then $\mult_{\widetilde{\sigma}(\widetilde{\bf t})}(C_{{\bf t}})=\mult_0(\widetilde{C})$ and $\mu(\widetilde{C}_{\widetilde{\bf t}})=\mu(\widetilde{C})$ for all $\widetilde{\bf t}\in\widetilde{T}$ due to the equisingularity of the deformation $\mathscr{Y}\to\widetilde{T}$.

For $\widetilde{\bf t}\in\widetilde{T}$, define
$$
f_{\widetilde{\bf t}}(x,y)=(x-\widetilde{\sigma}_1(\widetilde{\bf t}))^m\widetilde{f_{\widetilde{\bf t}}}\biggl(x,\frac{y-\widetilde{\sigma}_2(\widetilde{\bf t})}{x-\widetilde{\sigma}_1(\widetilde{\bf t})}+\widetilde{\sigma}_2(\widetilde{\bf t})\biggr),
$$
and $F(x,y,\widetilde{{\bf t}})=f_{\widetilde{\bf t}}(x,y)$. From \eqref{eqpi}, it follows that $F(x,y,\widetilde{\bf t})\in\bb{C}\{x,y,\widetilde{{\bf t}}-\widetilde{t}_0\}$.

Let
$$
\mathscr{Z}=\{((x,y),\widetilde{\bf t})\in(\bb{C}^2,0)\times(\widetilde{T},\widetilde{t}_0)\,:\, F(x,y,\widetilde{\bf t})=0\,\}
$$
and let $\psi:\mathscr{Z}\to\widetilde{T}$ be the morphism induced by the natural projection $(\bb{C}^2,0)\times(\widetilde{T},\widetilde{t}_0)\to(\widetilde{T},\widetilde{t}_0)$.

We claim that $\psi$ is flat. Indeed, since $\mathscr{Z}$ is a hypersurface, it is Cohen-Macaulay; moreover, $\widetilde{T}$ is smooth and each fiber of $\psi$ has dimension $1=\dim\mathscr{Z}-\dim\widetilde{T}$; the flatness follows immediately by \cite[Proposition 1.85, p.\,88]{GLS}.

Therefore, $\mathscr{Z}\xrightarrow{\psi}\widetilde{T}$ is a deformation of $\psi^{-1}(\widetilde{t}_0)\simeq C$.

Let $C_{\widetilde{\bf t}}\subset\bb{C}^2$ be the plane curve defined by equation $f_{\widetilde{\bf t}}(x,y)=0$. Then $\widetilde{\sigma}(\widetilde{\bf t})$ is a singular point of $C_{\widetilde{\bf t}}$ of multiplicity $m$; in addition, $\widetilde{C}_{\widetilde{\bf t}}$ is clearly a blowup of $C_{\widetilde{\bf t}}$. Hence by \cite[Theorem 6.58, p.\,149]{Wal}, we have
\begin{eqnarray*}
\mu(C_{\widetilde{\bf t}})&=&\mu(\widetilde{C}_{\widetilde{\bf t}})+m(m-1)\\
                          &=&\mu(\widetilde{C})+m(m-1)\\
                          &=&\mu(C),
\end{eqnarray*}
and thus the deformation $\mathscr{Z}\to\widetilde{T}$ is equisingular; see \cite[Corollary 2.68, p.\,371]{GLS}.
It is clear that $\mathscr{Y}$ is isomorphic to the equisingular blowup of $\mathscr{Z}$ along $\widetilde{\sigma}$.

{\bf Step 3: Versality of $\widetilde{\mathscr{C}}$. }
Recall that $\mathscr{C}\to T$ is a miniversal equisingular deformation of $C$. Form its versality, the equisingular deformation $\mathscr{Z}\to\widetilde{T}$ can be obtained as a pullback of $\mathscr{C}\to T$ under a morphism $\rho: (\widetilde{T},\widetilde{t}_0)\to(T,t_0)$. Since $\mathscr{Y}$ is isomorphic to an equisingular blowup of $\mathscr{Z}$, it is isomorphic to the pull back by $\rho$ of the equisingular blowup of $\mathscr{C}$, which is exactly $\widetilde{\mathscr{C}}$. Now the versality of $\widetilde{\mathscr{C}}\to T$ follows from that of $\mathscr{Y}\to\widetilde{T}$.
\end{proof}

\section{Relative K\"ahler differentials and minimal Tjurina numbers}\label{sec4}
Let $(C,0)\subset(\bb{C}^2,0)$ be a (reducible and irreducible) curve singularity germ of multiplicity $m$. Let $(\mathscr{C},x_0)\xrightarrow{\phi}(T,t_0)$ be a miniversal equisingular deformation of $C$. As in the previous section, we have morphisms
$$
\ov{\mathscr{C}}\xrightarrow{\widetilde{\nu}}\widetilde{\mathscr{C}}\xrightarrow{\pi}\mathscr{C},\qquad \nu=\pi\circ\widetilde{\nu},
$$
and
$$
\ov{\mathscr{C}}\xrightarrow{\widetilde{\nu}}\widetilde{\mathscr{C}}\xrightarrow{\widetilde{\phi}}T,\qquad\ov{\phi}=\widetilde{\phi}\circ\widetilde{\nu}.
$$

\subsection{Relative K\"ahler differentials}
Let $\Omega_{\mathscr{C}/T}$ (resp. $\Omega_{\widetilde{\mathscr{C}}/T}$, resp. $\Omega_{\ov{\mathscr{C}}/T}$) denote the $\mathscr{O}_{\mathscr{C}}$-module of relative K\"ahler differentials of $\mathscr{C}$ over $T$ (resp. the $\mathscr{O}_{\widetilde{\mathscr{C}}}$-module of relative K\"ahler differentials of $\widetilde{\mathscr{C}}$ over $T$, resp. the $\mathscr{O}_{\ov{\mathscr{C}}}$-module of relative K\"ahler differentials of $\ov{\mathscr{C}}$ over $T$). Define
$$
\mathscr{F}=\widetilde{\nu}^*\Omega_{\widetilde{\mathscr{C}}/T}/\nu^*\Omega_{\mathscr{C}/T},
$$
then $\mathscr{F}$ is a coherent sheaf of $\mathscr{O}_{\ov{\mathscr{C}}}$-modules and it is zero if the curve germ $C$ is smooth.

When $\mathscr{F}$ is not zero, its support $\Supp(\mathscr{F})$ is finite over $T$, hence $\ov{\phi}_*\mathscr{F}$ is a finite $\mathscr{O}_T$-module, and thus we can define
$$
\mathscr{D}({\bf t})=\text{length}_{\mathscr{O}_{T,{\bf t}}}((\ov{\phi}_*\mathscr{F})_{{\bf t}})=\sum_{p\in\ov{\phi}^{-1}({\bf t})\cap\Supp(\mathscr{F})}\dim\frac{\mathscr{F}_p}{\mathfrak{m}_p\mathscr{F}_p},\qquad\text{for } {\bf t}\in T.
$$
When $\mathscr{F}$ is zero, we define $\mathscr{D}({\bf t})\equiv0$; this is the case if the curve germ $C$ is smooth.

By the semicontinuity Theorem, see \cite[Theorem 1.81, p.\,84]{Wal}, it follows that the function ${\bf t}\mapsto \mathscr{D}({\bf t})$ is upper-semicontinuous on $T$ in the classical topology. Thus, there exists $\mathscr{D}_{\min}\in\bb{N}$ and an open dense subset $T_1\subset T$ with respect to the classical topology such that $\mathscr{D}_{\min}=\min\limits_{{\bf t}\in T}\mathscr{D}(t)$ and
$$
\mathscr{D}({\bf t})=\mathscr{D}_{\min} \text{ if and only if }{\bf t}\in T_1.
$$
Moreover, we have
$$
\mathscr{D}({\bf t})=\dim\frac{\widetilde{\nu}^*\Omega_{\widetilde{C}_{{\bf t}}}}{\nu^*\Omega_{C_{{\bf t}}}},
$$
where $C_{{\bf t}}=\phi^{-1}({\bf t})$ and $\widetilde{C}_{{\bf t}}=\widetilde{\phi}^{-1}({\bf t})$; it follows from \eqref{taudif} that
$$
\mathscr{D}({\bf t})=(\tau(C_{{\bf t}})-\tau(\widetilde{C}_{{\bf t}}))-\frac{m(m-1)}{2}.
$$
thus the function ${\bf t}\mapsto\tau(C_{{\bf t}})-\tau(\widetilde{C}_{{\bf t}})$ is also upper semi-continuous, and
\begin{equation}\label{maxtaudif}
\min_{{\bf t}\in T}(\tau(C_{{\bf t}})-\tau(\widetilde{C}_{{\bf t}}))=\frac{m(m-1)}{2}+\mathscr{D}_{\min},
\end{equation}
and the minimum is attained at ${\bf t}$ if and only if ${\bf t}\in T_1$. In particular, we have
\begin{equation}\label{esttaudif}
\tau(C)-\tau(\widetilde{C})\geq\frac{m(m-1)}{2}+\mathscr{D}_{\min},
\end{equation}
because $\phi^{-1}(t_0)=C$ and $\widetilde{\phi}^{-1}(t_0)=\widetilde{C}$. Note that $\mathscr{D}_{\min}=0$ when $C$ is smooth.

\subsection{Minimal Tjurina numbers}
Since Tjurina number is upper semicontinuous, see \cite[Theorem 2.6, p.\,114]{Wal}, there exists two numbers $\tau_{\min},\widetilde{\tau}_{\min}\in\bb{N}$ and an open dense subset $T_2\subset T$ with respect to the classical topology such that $\min\limits_{{\bf t}\in T}\tau(C_{{\bf t}})=\tau_{\min}$ and $\min\limits_{{\bf t}\in T}\tau(\widetilde{C}_{{\bf t}})=\widetilde{\tau}_{\min}$, and
\begin{equation}\label{taumin}
\tau(C_{{\bf t}})=\tau_{\min},\text{ and } \tau(\widetilde{C}_{{\bf t}})=\widetilde{\tau}_{\min},\ \forall\,{\bf t}\in T_2.
\end{equation}
Let $T^*=T_1\cap T_2$, then for all $t\in T^*$, we have from \eqref{taumin} and \eqref{maxtaudif} that
\begin{equation}\label{taumindif}
\tau_{\min}-\widetilde{\tau}_{\min}=\frac{m(m-1)}{2}+\mathscr{D}_{\min}.
\end{equation}

Since the equisingular deformations $\mathscr{C}\to T$ and $\widetilde{\mathscr{C}}\to T$ are both versal, the numbers $\tau_{\min}$ and $\widetilde{\tau}_{\min}$ coincide with the minimal Tjurina numbers studied in \cite{GHe}. From Theorem 13 in \cite{GHe} and \eqref{taumindif}, the following lemma follows.
\begin{lem}[\cite{GHe}, Theorem 13]
Let $m$ denote the multiplicity of $(C,0)$ and suppose $m\geq 2$. Then the minimal number $\mathscr{D}_{\min}$ is given by
$$
\mathscr{D}_{\min}=\frac{m(m-1)}{2}-\biggl\{\biggl(\biggl[\frac{m}{2}\biggr]-1\biggr)\biggl(m-\biggl[\frac{m}{2}\biggr]\biggr)+1-p_1(C)\biggr\},
$$
where $p_1(C)$ is an invariant depending on the Puiseux characteristic of $C$ and satisfies the following property:
$$
\left\{\begin{aligned}
p_1(C)&\geq 1,\ \text{for }m \text{ even},\\
p_1(C)&\geq 0,\ \text{for } m \text{ odd}.
\end{aligned}\right.
$$
\end{lem}

Hence we have the following estimate on $\mathscr{D}_{min}$.

\begin{claim}\label{estDmin}
If the multiplicity $m\geq 2$, we have
$$
\mathscr{D}_{\min}>\frac{m(m-1)}{4}.
$$
\end{claim}
\begin{proof}
Denote by
$$
\Delta=\frac{m(m-1)}{4}-\biggl\{\biggl(\biggl[\frac{m}{2}\biggr]-1\biggr)\biggl(m-\biggl[\frac{m}{2}\biggr]\biggr)+1-p_1(C)\biggr\},
$$
and we need to show $\Delta>0$. Our proof will be divided into two cases.

{\bf Case 1: the multiplicity $m$ is even. } We have $p_1(C)\geq 1$ and
\begin{eqnarray*}
\Delta&\geq&\frac{m(m-1)}{4}-\biggl(\frac{m}{2}-1\biggr)\biggl(m-\frac{m}{2}\biggr)\\
      &=&\frac{m(m-1)}{4}-\frac{m(m-2)}{4},
\end{eqnarray*}
hence $\Delta>0$.

{\bf Case 2: the multiplicity $m\geq 3$ is odd. } We have $p_1(C)\geq 0$ and
\begin{eqnarray*}
\Delta&\geq&\frac{m(m-1)}{4}-\biggl(\frac{m-1}{2}-1\biggr)\biggl(m-\frac{m-1}{2}\biggr)-1\\
      &=&\frac{m(m-1)}{4}-\frac{(m-3)(m+1)}{4}-1\\
      &=&\frac{m+3}{4}-1,
\end{eqnarray*}
therefore we also have $\Delta>0$.
\end{proof}

\subsection{Proof of Theorem \ref{thm1} }
By \eqref{esttaudif}, we have
\begin{eqnarray*}
 & &3\mu(\widetilde{C})-4\tau(\widetilde{C})-(3\mu(C)-4\tau(C))\\
 &=&4(\tau(C)-\tau(\widetilde{C}))-3(\mu(C)-\mu(\widetilde{C}))\\
 &\geq& 4\mathscr{D}_{\min}-m(m-1);
\end{eqnarray*}
Theorem \ref{thm1} follows immediately from Claim \ref{estDmin}.

Now we can give a proof of the irreducible case of Question \ref{ques1}. Suppose $(C,0)$ is a reduced and irreducible plane curve germ. There exists a minimal resolution of $C$ by successive embedded blowups
$$
C^{(N)}\to C^{(N-1)}\to\cdots\to C^{(1)}\to C^{(0)}=C
$$
where $C^{(N)}$ is smooth while for $i<N$, $C^{(i)}$ is singular; in addition, for $0\leq i<N$, $C^{(i+1)}$ is the strict transform of $C^{(i)}$ under a single blowup; see for instance, \cite[Chapter 3, p.\,43]{Wal}. From Theorem \ref{thm1}, it follows that
$$
3\mu(C^{(i)})-4\tau(C^{(i)})<3\mu(C^{(i+1)})-4\tau(C^{(i+1)}),\ \text{ for }0\leq i\leq N-1.
$$
In particular, we have
$$
3\mu(C)-4\tau(C)=3\mu(C^{(0)})-4\tau(C^{(0)})<3\mu(C^{(N)})-4\tau(C^{(N)})=0,
$$
that is, $\frac{\mu(C)}{\tau(C)}<\frac{4}{3}$.

\section{Comparison of singularities}
Consider a class of singularities $\mathscr{S}$. For instance, the class of isolated hypersurfaces singularities, the class of isolated complete intersection singularities, the class of irreducible curve singularities, and so on.

\begin{definition}
For two singularities $X',X\in\mathscr{S}$, we say $X'$ is \emph{smoother} than $X$, or $X$ is \emph{more singular} than $X'$ if $X'$ can be obtained from $X$ by a sequence of successive blowups.
\end{definition}

Let $I:\mathscr{S}\to\bb{R}$ be a function that associates the same value to analytically isomorphic members of $\mathscr{S}$. We call $I$ a numerical invariant for the class $\mathscr{S}$.

\begin{ex}
Let $\mathscr{S}$ be the class of reduced, irreducible plane curve germs. For $C\in\mathscr{S}$, define $I(C)=\mu(C)$. Then $I=\mu$ is a numerical invariant.

Similarly, the multiplicity and Tjurina number $\tau$ are also numerical invariants for $\mathscr{S}$. Thus $3\mu-4\tau$ is also a numerical invariant for $\mathscr{S}$.
\end{ex}

\begin{definition}
Suppose $I$ is a numerical invariant for the class $\mathscr{S}$.

We call $I$ an \emph{increasing ({\rm resp.} decreasing) invariant under blowups} if for all $X\in\mathscr{S}$, letting $\widetilde{X}\to X$ be the blowup of $X$ along some closed subvariety, we have
\begin{equation}\label{eq:star}
I(\widetilde{X})\geq I(X) (\text{resp. }I(\widetilde{X})\leq I(X)).
\end{equation}
We say that $I$ is \emph{strictly increasing ({\rm resp.} strictly decreasing)} if the inequality \eqref{eq:star} is strict whenever $X$ is not smooth and the blowup is one along some closed subvariety contained in the singular locus of $X$.
\end{definition}

Milnor number ad Tjurina number are useful decreasing invariants under blowups for isolated singularities. Our Theorem \ref{thm1} can be restated as follows.

\begin{cor}
For reduced irreducible plane curve singularity germs, the invariant $3\mu-4\tau$ is strictly increasing under blowups.
\end{cor}

Clearly, if $I$ is a strictly increasing numerical invariant under blowups and $I(Y)\leq I(X)$ for $X,Y\in\mathscr{S}$, then $Y$ cannot be smoother than $X$. Thus, monotonic numerical invariants under blowups can indeed tell something useful about comparison of singularities in an easy way. Even for the class of plane curve singularities, it is not easy to tell if one given curve is smoother than another or not. Milnor numbers and Tjurina numbers are easy to compute, but are not very effective; although Puiseux characteristics are very effective, they are not easily computable.

In the following examples, we will compute and compare the Milnor number $\mu$, Tjurina number $\tau$ and the invariant $3\mu-4\tau$ of a curve germ at $0$.  The computations are performed using {\sc Singular} software; see \cite{DGPS}.

\begin{ex}
For the curves $C: x^{11}+y^{11}+x^6y^6=0$ and $C': x^9+y^9+x^6y^6=0$, we have
$$
\mu(C)=100,\ \tau(C)=84,\ 3\mu(C)-4\tau(C)=-36,
$$
and
$$
\mu(C')=64,\ \tau(C')=60,\ 3\mu(C')-4\tau(C')=-48.
$$
So $C'$ is not smoother than $C$.
\end{ex}

\begin{ex}
For the curves $C: x^{13}+y^{12}+x^6y^7=0$ and $C': x^{11}+y^{10}+x^6y^6=0$, we have
$$
\mu(C)=132,\ \tau(C)=108,\ 3\mu(C)-4\tau(C)=-36,
$$
and
$$
\mu(C')=90,\ \tau(C')=78,\ 3\mu(C')-4\tau(C')=-42.
$$
So $C'$ is not smoother than $C$.
\end{ex}


\bigskip
\begin{thebibliography}{00}

      \bibitem{AlB} P.~Almir\'on, G.~Blanco, \emph{A note on a question of Dimca and Greuel}, C. R. Math. Acad. Sci. Paris,Ser. I {\bf 357} (2019), 205--208.

       \bibitem{Alv} E.~Casas-Alvero, \emph{Singularities of plane curves}, London Math. Soc. Lecture Note Ser. {\bf 276}, Cambridge
Univ. Press, Cambridge, 2000.

      \bibitem{DGPS} W.~Decker, G.-M.~Greuel, G.~Pfister, H.~Sch\"onemann, \emph{{\sc Singular} {4-1-2}\,---\,A computer algebra system for polynomial computations},
http://www.singular.uni-kl.de, (2019).

	 \bibitem{DGr} A.~Dimca, G.-M.~Greuel, \emph{On 1-forms on isolated complete intersection on curve singularities}, J. of
Singul. {\bf 18} (2018), 114–-118.

      \bibitem{Gen} Y.~Genzmer, \emph{Dimension of the moduli space of a curve in the complex plane}, Preprint in:
arXiv:1610.05998 (2016).

      \bibitem{GHe} Y.~Genzmer, M.~E.~Hernandes, \emph{On the Saito's basis and the Tjurina Number for Plane Branches}, arXiv:1904.03645, 2019.

      \bibitem{GLS} G.-M.~Greuel, C.~Lossen, E.~Shustin, \emph{Introduction to Singularities and Deformations}, Springer Monographs
in Mathematics, Berlin, 2007.

		\bibitem{Hir} H.~Hironaka, \emph{Resolution of singularities of an algebraic variety over a field of
characteristic zero, I, II}, Ann. of Math. (2) {\bf 79} (1964), 109--203; ibid. (2) {\bf 79} (1964),
205-–326.

    \bibitem{Kin} H.~C.~King, \emph{Topological type in families of germs}, Invent. Math. {\bf 62} (1980), 1–-13.
	
	\bibitem{Kol} J.~Koll\'ar, \emph{Lectures on resolution of singularities,} Annals of Mathematics
Studies, {\bf 166}, Princeton University Press, Princeton, New Jersey, 2007.
	
    \bibitem{MPGM} M.~Alberich-Carrami$\tilde{\rm n}$ana, P.~Almir\'on, G.~Blanco and A.~Melle-Hern\'andez, \emph{The minimal tjurina number of irreducible germs of plane curve singularities},
arXiv:1904.02652, 2019.

  \bibitem{MT} J.~Morgan, G.~Tian, \emph{Ricci Flow and the Poincar\'e Conjecture}, Clay Mathematics Institute, 2007.

	\bibitem{Mum} D.~Mumford, \emph{Algebraic Geometry I: Complex Projective Varieties}, Classics in Math., Springer-
Verlag, Berlin, 1995.
	

    \bibitem{Tei} B.~Teissier, Appendix, in \cite{Zar}, 1986.

    \bibitem{Tim} J.~G.~Timourian, \emph{The invariance of Milnor's number implies topological triviality},
Amer. J. Math. {\bf 99} (1977), 437-–446.


    \bibitem{TrR} L.~D.~Tr\'ang, C.~P.~Ramanujam, \emph{The invariance of Milnor's number implies the invariance of the
topological type}, Amer. J. Math. {\bf 98} (1976), 67-–78.

	\bibitem{Wal} C.~T.~W.~Wall, \emph{Singular points of plane curves}, London Mathematical Society Student
Texts, {\bf 63}, Cambridge University Press, Cambridge, 2004.
	
     \bibitem{Zar} O.~Zariski, \emph{Le probl\`eme des modules pour les branches planes}, Hermann, Paris, 1986.
\end{thebibliography}
\end{document}